\documentclass[letterpaper, 10 pt, journal, twoside]{IEEEtran}  

\IEEEoverridecommandlockouts                              

\usepackage{amsmath,amssymb,amsthm} 

\newtheorem{theorem}{Theorem}[section]

\newtheorem{lemma}[theorem]{Lemma}
\newtheorem{proposition}[theorem]{Proposition}
\newtheorem*{remark}{Remark}

\usepackage{graphicx} 
\usepackage{epsfig} 
\usepackage{epstopdf}
\usepackage{mathtools}

\usepackage{cite}
\usepackage[ruled,vlined]{algorithm2e}

\usepackage{tikz}
\usetikzlibrary{shapes, arrows,positioning, calc, chains}

\makeatletter
\DeclareRobustCommand{\rvdots}{%
	\vbox{
		\baselineskip4\p@\lineskiplimit\z@
		\kern-\p@
		\hbox{.}\hbox{.}\hbox{.}
}}
\makeatother

\usepackage{lipsum}
\makeatletter
\def\footnoterule{\relax%
	\kern-5pt
	\hbox to \columnwidth{\hfill\vrule width .9\columnwidth height 0.4pt\hfill}
	\kern4.6pt}
\makeatother

\usepackage{color,hyperref}
\definecolor{darkblue}{rgb}{0.0,0.0,0.6}
\hypersetup{colorlinks,breaklinks,linkcolor=darkblue,urlcolor=darkblue,anchorcolor=darkblue,citecolor=darkblue}

\makeatother
\title{On Risk-Sensitive
Decision  \\ Making Under Uncertainty}

\author{\large Chung-Han Hsieh,$^{*}$ \textit{Member, IEEE} and Yi-Shan Wong${}^{**}$
	\thanks{This paper is partially supported by the Ministry of Science and Technology~(MOST), Taiwan. under Grants:  MOST111-2221-E-007-124-.\\ ${}^*$Chung-Han Hsieh is with the Department of Quantitative Finance, National Tsing Hua University, Hsinchu 300044, Taiwan. E-mail: \href{mailto: ch.hsieh@mx.nthu.edu.tw}{ch.hsieh@mx.nthu.edu.tw}. } 
	\thanks{\hskip -10pt ${}^{**}$Yi-Shan Wong was with the Department of Quantitative Finance, National Tsing Hua University, Hsinchu 300044, and is currently with the Actuarial Department I, Cathay Life Insurance, Taipei 106436, Taiwan. (E-mail: \href{matilto: yishan.wong13@gmail.com}{yishan.wong13@gmail.com}).}
}

\begin{document}

	\maketitle
	\thispagestyle{empty}
	\pagestyle{empty}
	
	
\begin{abstract} 
    This paper studies a {\em risk-sensitive} decision-making problem under uncertainty. It considers a decision-making process that unfolds over a fixed number of stages, in which a decision-maker chooses among multiple alternatives, some of which are deterministic and others are stochastic. The decision-maker's cumulative value is updated at each stage, reflecting the outcomes of the chosen alternatives. After formulating this as a stochastic control problem, we delineate the necessary optimality conditions for it. Two illustrative examples from optimal betting and inventory management are provided to support our theory.
\end{abstract}
	
\vspace{3mm}
	
	
\section{Introduction}\label{section: introduction}
Optimal decision-making under uncertainty is critical across various disciplines, including engineering, management, economics, and finance. The standard approach to addressing this issue is by applying expected utility theory, which formulates the problem as stochastic programming; see \cite{shapiro2021lectures, bertsekas2012dynamic}.

Early paradigms for this type of problem include expected utility theory in \cite{von2007theory}, the Markowitz mean-variance portfolio optimization \cite{Markowitz_1952},  Kelly's criterion \cite{Kelly_1956, thorp2006kelly}, and the prospect theory~\cite{kahneman1979prospect}. Subsequent studies have explored different aspects of this topic. For instance, \cite{kuhn2010analysis} analyzed the rebalancing frequency in log-optimal portfolio selection, \cite{wu2020analysis} examined the Kelly betting problem in a finite stage framework. Studies by \cite{luenberger1993preference, luenberger2013investment} investigated the foundations of using Expected Logarithmic Growth (ELG) and the variance of logarithmic growth as preferences in portfolio management problems. 

Further studies by \cite{algoet1988asymptotic, cover2006elements} delved into ELG portfolios and derived various optimalities. More recently, \cite{hsieh2020necessary, hsieh2023asymptotic} explored the frequency-dependent log-optimal portfolio and its asymptotic log-optimality. Then
\cite{wong2023frequency} studied the frequency-dependent log-optimal portfolio problems with costs. A nonlinear control theoretic approach is considered in \cite{proskurnikov2023benefit}.
However, these studies largely omitted an explicit focus on \textit{risk} within the optimization goal.

To this end, this paper studies a class of \textit{risk-sensitive} decision-making problems under uncertainty. The objective is composed of expected log-growth and variance of log-growth of the decision maker's account.
We consider a decision-making process that unfolds over a fixed number of stages, where a decision-maker chooses among multiple alternatives, some of which are deterministic and others are stochastic. The decision-maker's accumulated value is updated at each stage based on the chosen alternatives. This problem is modeled as a stochastic control problem, aimed at maximizing a risk-sensitive objective function. 
Furthermore, the paper provides the necessary optimality conditions for this problem.

 The rest of the paper is organized as follows. Section~\ref{section: problem formulation} provides the problem formulation for the risk-sensitive decision-making problem. Section~\ref{section: main results} first presents an equivalent optimization problem.
Then we provide rigorous proof for the necessary condition for the optimality of the problem. Subsequently, Section~\ref{section: Illustrative Examples} illustrates the optimality result with several numerical examples. Finally, Section~\ref{section: Concluding Remarks} provides concluding remarks and outlines potential future research directions.

\section{Problem Formulation}\label{section: problem formulation}	
Consider a decision maker engaged in a decision-making process that unfolds over a fixed number of stages, denoted by $N>1$. 
For $k=0,1,\dots, N-1$, a decision maker chooses among~$m \geq 2$ alternatives, with one being a deterministic option with a nonnegative payoff $r(k) \geq 0.$ That is, the outcome is certain and is treated as a degenerate random variable with value $r(k)$ for all $k$ with probability one. 
Alternatively, if the~$i$th alternative is stochastic, we take~$X_i(k)$ as the random payoffs.
We assume that the payoff vectors~$
X(k):=\left[ X_1(k), \, X_2(k),\,  \dots, \,X_m(k) \right]^\top 
$
are i.i.d. and have a known distribution and have components $X_i(\cdot)$ which can be arbitrarily correlated. 
These vectors have components satisfying
$
X_{\min,i}  \leq X_i(k) \leq X_{\max,i}
$
with predefined limits, where $-1 < X_{\min,i} < 0 < X_{\max,i} < 1$.

\subsection{Notation}
    Throughout the paper, we take $\mathbb{R}$ to be the set of real numbers and $\mathbb{R}_{+}$ to be the set of nonnegative real numbers.
    We denote by $\textbf{1}$ the vector with all components to be equal to~$1$ and denote by $ e_i \in \mathbb{R}^m $ to be the unit vector having~1 at the~$i$th component.
    All random objects are defined in a probability space~$(\Omega, \mathcal{F}, \mathbb{P})$ with $\Omega$ being the sample space,~$\mathcal{F}$ being the information set, and $\mathbb{P}$ being the probability measure. 
    For~$a,b \in \mathbb{R}^n$, the term $\langle a, b \rangle$ means the standard inner product~$a^\top b$.

\subsection{Linear Decision Policy}
Take $V(k)$ to be the decision maker's accumulated value and the $i$th feedback gain 
	$
    K_i(k) \in [0, 1]
	$
represents the fraction of the account allocated to the~$i$th asset for~$i=1,\dots,m$. 
Said another way,  the~$i$th controller is  a linear feedback of the form~$
	u_i(k) = K_i(k) V(k).
	$
For each $k$, we take $K_i(k) := K_i$ and impose a unit simplex constraint
	$$
	K \in {\mathcal K} := \left\{K \in \mathbb{R}^{m}: K_i \geq 0 \text{ for all $i$}, \; K^\top \textbf{1} = 1 \right\}
	$$
 which is readily verified as a convex set.

\subsection{Value Dynamics} 
    Letting $n \geq 1$ be the number of steps between decisions,  at time~{$k=0$}, the decision-maker begins with initial account~$V_0 := V(0)$ and initial investment control~{$
	u(0) = \sum_{i=1}^m K_i V_0
	$}
	and waits~$n$ steps. Then, when~{$k=n$}, the control is updated to be~$
		u(n) = \sum_{i=1}^m K_i V(n).
		$

The decision-maker starts with an initial accumulated value~$V(0)>0$. At each stage, the decision-maker chooses to allocate a fraction of the accumulated value to each of the~$m$ alternatives. This allocation is represented by the control gain~$K$, where each $K_i$ is the fraction of the accumulated value allocated to the $i$th alternative.
For decision period~$n \geq 1$, the corresponding account value at stage $n$ is described by the stochastic recursion
	$
	V(n) =  \langle K, \mathcal{R}_n \rangle V_0 
	$
where $\langle K, \mathcal{R}_n \rangle$ represents the total return from all alternatives at stage $n$ with the random vector $\mathcal{R}_n$ having the $i$th component 
	$$
	\mathcal{R}_{n,i} := \prod_{k=0}^{n-1} (1+X_i(k)) 
	$$
 for $i=1,2, \ldots, m$.
It is readily seen that~{$\mathcal{R}_{n,i} > 0$} for all~$n \geq 1$.
In the sequel, we may sometimes write $V(n; K)$ to emphasize the dependence on the feedback gain $K$.

\subsection{Risk-Sensitive Decision Making Problem}
For any decision period~$n \geq 1$, we consider a  \textit{risk-sensitive} objective function,\footnote{In finance, a utility function $U$ is said to be \textit{risk-averse} if it is concave, continuous, and strictly increasing; see \cite{luenberger2013investment}. } see \cite{luenberger1993preference}, as follows:
Take
\begin{align} \label{eq: risk-sensitive objective}
	{U_n^\rho}(K; X) 
	&:= \frac{1}{n} \mathbb{E}\left[ \log \frac{V(n; K)}{V_0} \right] - \frac{\rho}{2n^2}\text{var}\left( \log \frac{V(n; K)}{V_0} \right)
\end{align}
where $\rho \geq 0$ is the \textit{risk-aversion} constant, which modulates the degree of risk sensitivity, with higher values of $\rho$ indicating a greater aversion to risk.
Our goal is to solve the following stochastic control problem parameterized by feedback gain $K$:
\begin{align} \label{problem: stochastic optimal control problem}
    \begin{aligned}
        &\sup_{K \in \mathcal{K}}\; U_n^\rho (K; X)\\
        &{\rm s.t. } \; V(n) = \langle K, \mathcal{R}_n \rangle V_0 
    \end{aligned}
\end{align}
with $\mathcal{R}_n = [\mathcal{R}_{n,1}, \dots, \mathcal{R}_{n,m}]^\top$ with $ \mathcal{R}_{n,i} = \prod_{k=0}^{n - 1} ( 1 + X_i(k) ) $ for $i \in \{1, 2, \dots, m \}$.
That is, the decision-maker must balance the expected logarithmic return against the risk, represented by the variance of the logarithmic return.

\begin{remark}\rm
    It should be noted that Problem~\eqref{problem: stochastic optimal control problem} may be nonconvex since the log-variance quantity may not be convex. However, if the risk-aversion constant~$\rho =0$, then Problem~\eqref{problem: stochastic optimal control problem} reduces to the classical expected log-growth maximization problem, which generally forms a concave program; see \cite{maclean2011kelly,hsieh2023asymptotic}. Later in Section~\ref{section: Illustrative Examples}, we shall demonstrate some cases where the log-variance is indeed convex.
\end{remark}

\section{Main Results}\label{section: main results}
This section first presents an equivalent problem and then provides the necessary conditions for the optimality of the problem.
 Indeed, the next lemma simplifies Problem~\eqref{problem: stochastic optimal control problem}.

\begin{lemma} \label{lemma: equivalent risk-sensitive problem}
For $n\geq 1$,
    the risk-sensitive objective function~$U_n^\rho(K; X)$   satisfies
    \begin{align*}
	{U_n^\rho}(K; X) 
	&= \frac{1}{n}\mathbb{E}\left[ \log \langle K, \mathcal{R}_n \rangle  \right] - \frac{\rho}{2n^2}\mathbb{E}\left[ \left( \log\langle K, \mathcal{R}_n \rangle  \right)^2 \right] \\
	&\qquad + \frac{\rho}{2n^2}\left( \mathbb{E}\left[ \log\langle K, \mathcal{R}_n \rangle  \right] \right)^2.
\end{align*} 
\end{lemma}

\begin{proof}
We shall proceed with a straightforward calculation.
Note that
    \begin{align*}
	{U_n^\rho}(K; X) 
	&= \frac{1}{n} \mathbb{E}\left[ \log \frac{V(n)}{V(0)} \right] - \frac{\rho}{2n^2}\text{var}\left( \log \frac{V(n)}{V(0)} \right)  \\
	&= \frac{1}{n}\mathbb{E}\left[ \log \langle K, \mathcal{R}_n \rangle \right] - \frac{\rho}{2n^2}\text{var}\left( \log \langle K, \mathcal{R}_n \rangle \right) \\
	&= \frac{1}{n}\mathbb{E}\left[ \log \langle K, \mathcal{R}_n \rangle \right] - \frac{\rho}{2n^2}\mathbb{E}\left[ \left( \log \langle K, \mathcal{R}_n \rangle \right)^2 \right] \\
	&\qquad + \frac{\rho}{2n^2}\left( \mathbb{E}\left[ \log \langle K, \mathcal{R}_n \rangle \right] \right)^2,
\end{align*}
which is desired.
\end{proof}

While Problem~\eqref{problem: stochastic optimal control problem} may be nonconvex, see Appendix, the following lemma indicates that the necessary condition for optimality is possible to establish.

\begin{lemma}[Necessary Optimality Conditions] \label{Lemma: risk term}
    If the feedback vector $K^* \in \mathcal{K}$ is optimal to the stochastic control problem
\begin{align} 
    \begin{aligned} \label{problem: risk-sensitive problem}
        &\sup_{K \in \mathcal{K}}\; U_n^\rho (K; X)\\
        &{\rm s.t. } \; V(n) = \langle K, \mathcal{R}_n \rangle V_0 
    \end{aligned}
\end{align}
    then for~$i=1,2,\ldots,m$, we have
    \begin{align*}
        & \mathbb{E}\left[ \frac{ \mathcal{R}_{n,i} }{ \langle K^{*}, \mathcal{R}_n \rangle } \right] 
        - \frac{\rho}{n} \mathbb{E}\left[ \log \langle K^{*}, \mathcal{R}_n \rangle \cdot \frac{\mathcal{R}_{n,i}}{ \langle K^{*}, \mathcal{R}_n \rangle } \right] \\
        & +\frac{\rho}{n} \mathbb{E} \left[ \log \langle K^{*}, \mathcal{R}_n \rangle \right] \mathbb{E}\left[ \frac{\mathcal{R}_{n,i}}{\langle K^{*}, \mathcal{R}_n \rangle } \right] = 1, \text{ if $K_i^* > 0$ }
    \end{align*}
    \begin{align*}
         & \mathbb{E}\left[ \frac{\mathcal{R}_{n,i}}{\langle K^{*}, \mathcal{R}_n \rangle } \right] - \frac{\rho}{n}\mathbb{E}\left[ \log \langle K^{*}, \mathcal{R}_n \rangle \frac{\mathcal{R}_{n,i}}{\langle K^{*}, \mathcal{R}_n\rangle} \right] \\
        & +\frac{\rho}{n}\mathbb{E}\left[ \log \langle K^{*}, \mathcal{R}_n \rangle \right] \mathbb{E}\left[ \frac{\mathcal{R}_{n,i}}{ \langle K^{*}, \mathcal{R}_n \rangle } \right] \leq 1, \text{ if $K_i^* = 0$. }
    \end{align*}
\end{lemma}

\begin{proof}
  With the aid of Lemma~\ref{lemma: equivalent risk-sensitive problem},  we rewrite Problem~\eqref{problem: risk-sensitive problem} as an equivalent constrained minimization problem described as follows:
    \begin{align*}
        &\min_{K} -\mathbb{E}\left[ \log \langle K, \mathcal{R}_n \rangle \right] + \frac{\rho}{2n} \mathbb{E}\left[ (\log \langle K, \mathcal{R}_n \rangle )^2 \right] \\
        &\qquad   -\frac{\rho}{2n}\left( \mathbb{E}\left[ \log \langle K, \mathcal{R}_n \rangle \right]\right)^2 \\
        &{\rm s. t.} \; K^\top  \mathbf{1} - 1 = 0; \\
        &\;\;\;\;\; -K^\top  e_i \leq 0, \; i=1,2, \ldots, m
    \end{align*}
    where $ e_i \in \mathbb{R}^m $ is unit vector having 1 at the $i$th component. Consider the Lagrangian
    \begin{align*}
        \mathcal{L}(K, \lambda, \mu)
        &:= -\mathbb{E}\left[ \log \langle K, \mathcal{R}_n \rangle \right] + \frac{\rho}{2n} \mathbb{E} \left[ ( \log \langle K, \mathcal{R}_n \rangle)^2 \right] \\
        &\qquad -\frac{\rho}{2n}\left( \mathbb{E}\left[ \log \langle K, \mathcal{R}_n \rangle \right] \right)^2 \\ & \qquad + \lambda (K^\top  \mathbf{1} - 1) - \mu^\top K.
    \end{align*}
Then the Karush-Kuhn-Tucker (KKT) Conditions implies that if $K^*$ is a local maximum then there exists a scalar~$\lambda \in \mathbb{R}^1$ and a vector $\mu \in \mathbb{R}^m$ with component $\mu_j \geq 0$ such that, for~$i=1,2, \ldots, m$,
    \begin{align}
        &-\mathbb{E}\left[ \frac{\mathcal{R}_{n,i}}{ \langle K^{*}, \mathcal{R}_n \rangle} \right] \left( 1 + \frac{\rho}{n}\mathbb{E}\left[ \log \langle K^{*}, \mathcal{R}_n \rangle \right] \right) \nonumber \\
        &\qquad + \frac{\rho}{n}\mathbb{E}\left[ \log \langle K^{*}, \mathcal{R}_n \rangle \frac{\mathcal{R}_{n,i}}{\langle K^{*}, \mathcal{R}_n \rangle }  \right] +\lambda - \mu_i = 0 \label{eq:partial Ki (risk)} \\
        & K^{*\top} \mathbf{1} - 1 = 0 \label{eq: partical lambda (risk)} \\
        & \mu_i K_i^* = 0 \label{eq: Mu_i x Ki* = 0 (risk)}
    \end{align}
From Equation~\eqref{eq:partial Ki (risk)}, we obtain
    \begin{align} \label{eq: Mu_i with lambda (risk)}
        \mu_i 
        &= -\mathbb{E}\left[ \frac{\mathcal{R}_{n,i}}{K^{*\top}\mathcal{R}_n} \right] \left( 1 + \frac{\rho}{n}\mathbb{E}\left[ \log \langle K^{*}, \mathcal{R}_n \rangle \right] \right) \nonumber \\
        &\qquad +\frac{\rho}{n}\mathbb{E}\left[ \log \langle K^{*}, \mathcal{R}_n \rangle \cdot \frac{\mathcal{R}_{n,i}}{ \langle K^{*}, \mathcal{R}_n \rangle }  \right] +\lambda.
    \end{align}
for $i=1,2, \ldots, m$. Using the fact that $\mu_i K_i^* = 0$, for $i=1,2,\dots,n$, we take weighted sum of Equation~(\ref{eq: Mu_i with lambda (risk)}) and have
    \begin{align*}
        &\sum_{i=1}^m \mu_i K_i^* \\
        &\quad = \sum_{i=1}^m K_i^* \left( -\mathbb{E}\left[ \frac{\mathcal{R}_{n,i}}{ \langle K^{*}, \mathcal{R}_n \rangle } \right] \left( 1 + \frac{\rho}{n}\mathbb{E}\left[ \log \langle K^{*}, \mathcal{R}_n \rangle \right] \right) \right) \\
        &\quad + \sum_{i=1}^m K_i^* \left( \frac{\rho}{n} \mathbb{E} \left[ \log \langle K^{*}, \mathcal{R}_n\rangle \cdot \frac{ \mathcal{R}_{n,i} }{\langle K^{*}, \mathcal{R}_n \rangle}  \right] + \lambda \right) = 0. 
    \end{align*}
Since 
$$
\sum_{i=1}^m K_i^* \mathbb{E}\left[ \frac{\mathcal{R}_{n,i}}{\langle K^{*}, \mathcal{R}_n \rangle } \right] 
= \mathbb{E}\left[  \frac{\langle K^{*}, \mathcal{R}_n \rangle }{\langle K^{*}, \mathcal{R}_n \rangle} \right] = 1, 
$$
and $\sum_{i=1}^m K_i^* = 1$, we obtain
    \begin{align}
        -1 - \frac{\rho}{n} \mathbb{E}\left[ \log \langle K^{*}, \mathcal{R}_n \rangle \right] + \frac{\rho}{n} \mathbb{E}\left[ \log \langle K^{*}, \mathcal{R}_n \rangle \right] +\lambda = 0.
    \end{align}
Thus, we have $\lambda=1$ and substitute it into Equation~(\ref{eq: Mu_i with lambda (risk)}). This tells us that for $i=1,2, \ldots, m$,
    \begin{align*}
        \mu_i 
        &= -\mathbb{E}\left[ \frac{\mathcal{R}_{n,i}}{\langle K^{*}, \mathcal{R}_n \rangle } \right] \left( 1 + \frac{\rho}{n} \mathbb{E}\left[ \log \langle K^{*}, \mathcal{R}_n \rangle \right] \right) \nonumber \\
        &\qquad + \frac{\rho}{n}\mathbb{E}\left[ \log \langle K^{*}, \mathcal{R}_n \rangle  \cdot \frac{\mathcal{R}_{n,i}}{\langle K^{*}, \mathcal{R}_n \rangle }  \right] + 1.
    \end{align*}
    The fact $\mu_i K_i^*=0$ implies that if $K_i^* > 0$, then $\mu_i = 0$ and
    \begin{align*}
        & \mathbb{E}\left[ \frac{\mathcal{R}_{n,i}}{ \langle K^{*}, \mathcal{R}_n \rangle } \right] 
        - \frac{\rho}{n} \mathbb{E}\left[ \log \langle K^{*}, \mathcal{R}_n \rangle \cdot \frac{\mathcal{R}_{n,i}}{\langle K^{*}, \mathcal{R}_n \rangle } \right] \\
        &\qquad + \frac{\rho}{n}\mathbb{E}\left[ \log \langle K^{*}, \mathcal{R}_n \rangle \right] \mathbb{E}\left[ \frac{\mathcal{R}_{n,i}}{\langle K^{*}, \mathcal{R}_n \rangle} \right] = 1.
    \end{align*}
    If $K_i^* = 0$, then $\mu_i \geq 0$, which implies that
    \begin{align*}
        & \mathbb{E}\left[ \frac{\mathcal{R}_{n,i}}{\langle K^{*}, \mathcal{R}_n \rangle } \right] - \frac{\rho}{n}\mathbb{E} \left[ \log \langle K^{*}, \mathcal{R}_n \rangle  \cdot \frac{\mathcal{R}_{n,i}}{\langle K^{*}, \mathcal{R}_n \rangle } \right] \\
        &\qquad +\frac{\rho}{n}\mathbb{E}\left[ \log(K^{*\top}\mathcal{R}_n) \right] \mathbb{E}\left[ \frac{\mathcal{R}_{n,i}}{K^{*\top}\mathcal{R}_n} \right] \leq 1.
    \end{align*}
Hence, the proof is complete.    
\end{proof}

\begin{remark}\rm
    While Lemma~\ref{Lemma: risk term} only provides necessary conditions for optimality, it is possible to construct an example that {\em extends} these to both necessary and sufficient conditions. For an illustration,  see Section~\ref{section: Illustrative Examples} to follow.
\end{remark}

\section{Illustrations} \label{section: Illustrative Examples}
This section provides two examples to validate our findings on the impact of the risk-aversion coefficient. The first is on optimal betting and the second is on retail inventory management.

\subsection{Optimal Betting}
    We now provide an example in optimal betting and illustrate how the risk term~$\rho$ affects the optimal feedback vector~$K^*$. We consider a two-alternative decision-making process with one risk-less alternative having zero interest rate; i.e., $X_1(k)=0$ with probability one; and the other is a risky alternative with i.i.d. payoffs $X_2(k) \in \{ -\frac{1}{2},\frac{1}{2} \}$ with probability
    $$
    \mathbb{P}\left( X_2(k) = \frac{1}{2} \right) := p \in \left( \frac{1}{2}, \frac{3}{4} \right).
    $$
With this setting, the log-variance becomes
\begin{align*}
    v(K)
    & := {\rm var}\left( \log \frac{V(n; K)}{V_0} \right) \\
    & = \mathbb{E}\left[ \left( \log \langle K, \mathcal{R}_n \rangle  \right)^2 \right]  
        - \left( \mathbb{E} \left[ \log\langle K, \mathcal{R}_n \rangle  \right] \right)^2\\
    & = p (\log (1 + K_2/2))^2  + (1-p) \log (1- K_2/2 )^2  \\
    & \qquad - (p \log (1 + K_2/2) + (1-p)\log (1 - K_2/2)^2.
\end{align*}
A lengthy but straightforward calculation leads to the second-order derivative of $v(K)$:
\[
\frac{\partial^2 v(K)}{ \partial K_2^2 } = \frac{16 p (1-p)  (2 + K_2 \log \frac{2 + K_2 }{2 - K_2} )}{ \left( K_2^2 - 4 \right)^2} \geq 0,
\]
which shows that the log-variance $v(K)$ is convex. Therefore, the risk-sensitive objective~\eqref{eq: risk-sensitive objective} becomes a concave function in~$K_2$, which implies that Problem~\eqref{problem: risk-sensitive problem} is a concave program. Therefore, Lemma~\ref{Lemma: risk term} becomes a necessary and sufficient condition for optimality.    
Specifically, for~$n=1$, with the aids of Lemma~\ref{Lemma: risk term}, we have
\begin{align*}
        & \mathbb{E} \left[ \frac{ \mathcal{R}_{1,2} }{ \langle K^{*}, \mathcal{R}_1 \rangle } \right] - \rho\mathbb{E}\left[ \log \langle K^{*}, \mathcal{R}_1 \rangle \cdot \frac{ \mathcal{R}_{1,2} }{ \langle K^{*}, \mathcal{R}_1 \rangle } \right] \\
        & + \rho \mathbb{E} \left[ \log \langle K^{*}, \mathcal{R}_1 \rangle \right] \mathbb{E}\left[ \frac{\mathcal{R}_{1,2}}{ \langle K^{*}, \mathcal{R}_1 \rangle } \right] = 1.
    \end{align*}
We define the left-hand side as a function of $p, K_2$, and $\rho$, call it $f(p, K_2, \rho)$. Then, a straightforward calculation leads to
\begin{align*}
    f(p, K_2, \rho)
    & =-1 + \frac{2p}{1+K_2^*}-\frac{2p\rho\log( 1 + K_2^* )}{ 1 + K_2^*}\\
    & + \frac{2p\rho\left( p\log(1 + K_2^*)+(1-p)\log(1-K_2^*)\right)}{1+K_2^*}\\
    & =0.
\end{align*}
Note that, for example, if $\rho = 0$, the function $f$ has a zero-crossing solution $K_2^* = 2(2p - 1)$, which reduces to the classical Kelly's solution; see~\cite{Kelly_1956, hsieh2019contributions}.
As a second example, if~$p=0.6$, then $K_2^*=0.4$ as $\rho = 0$. If we increase~$\rho$ to~$0.1$, we obtain the~$K_2^* \approx 0.3646$, slightly smaller than~$0.4$, implying that investors will reduce the weight as the risk aversion constant $\rho$ rises. 
When $\rho$ rises from $0$ to $1$, $K_2^*$ falls by around $0.2$ to $0.2035$; see Figures~\ref{fig: small p} and \ref{fig: small p zoom in}.

We also consider extreme cases such as $p = 0.75$ which yields~$K_2^* = 1$ when $\rho = 0$ and $K_2^* \approx 0.5643$ when~$\rho = 1$; see Figure~\ref{fig: large p zoom in}. From the cases discussed above, we find that~$K_2^*$ is negatively affected by $\rho$ and is more sensitive to $\rho$ as the probability of profiting becomes higher.

\begin{figure}[h!]
    \centering
    \includegraphics[width=1\linewidth]{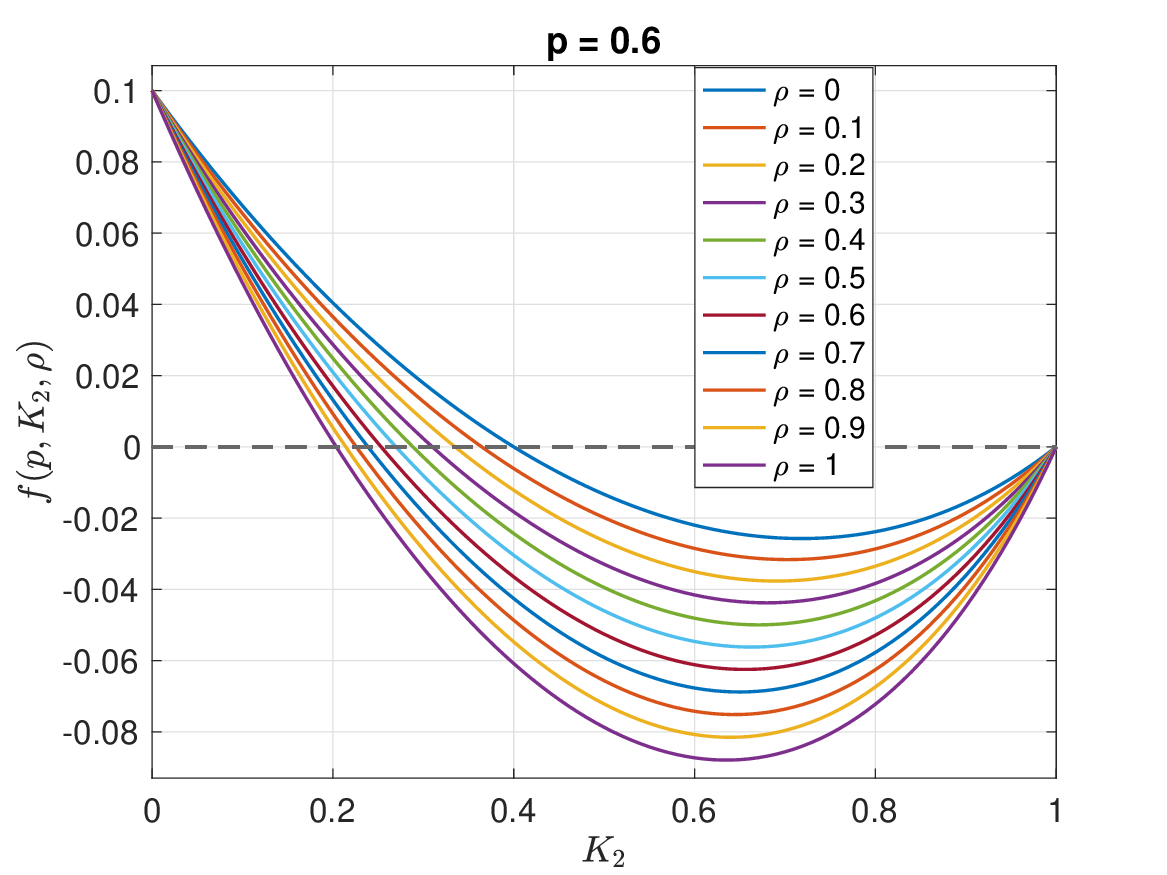}
    \caption{Optimal Feedback Gain $K_2^*$ for $\rho \in \{0, 0.1, 0.2, \dots, 1\}$ when $p=0.6$.}
    \label{fig: small p}
\end{figure}

\begin{figure}[h!]
    \centering
    \includegraphics[width=1\linewidth]{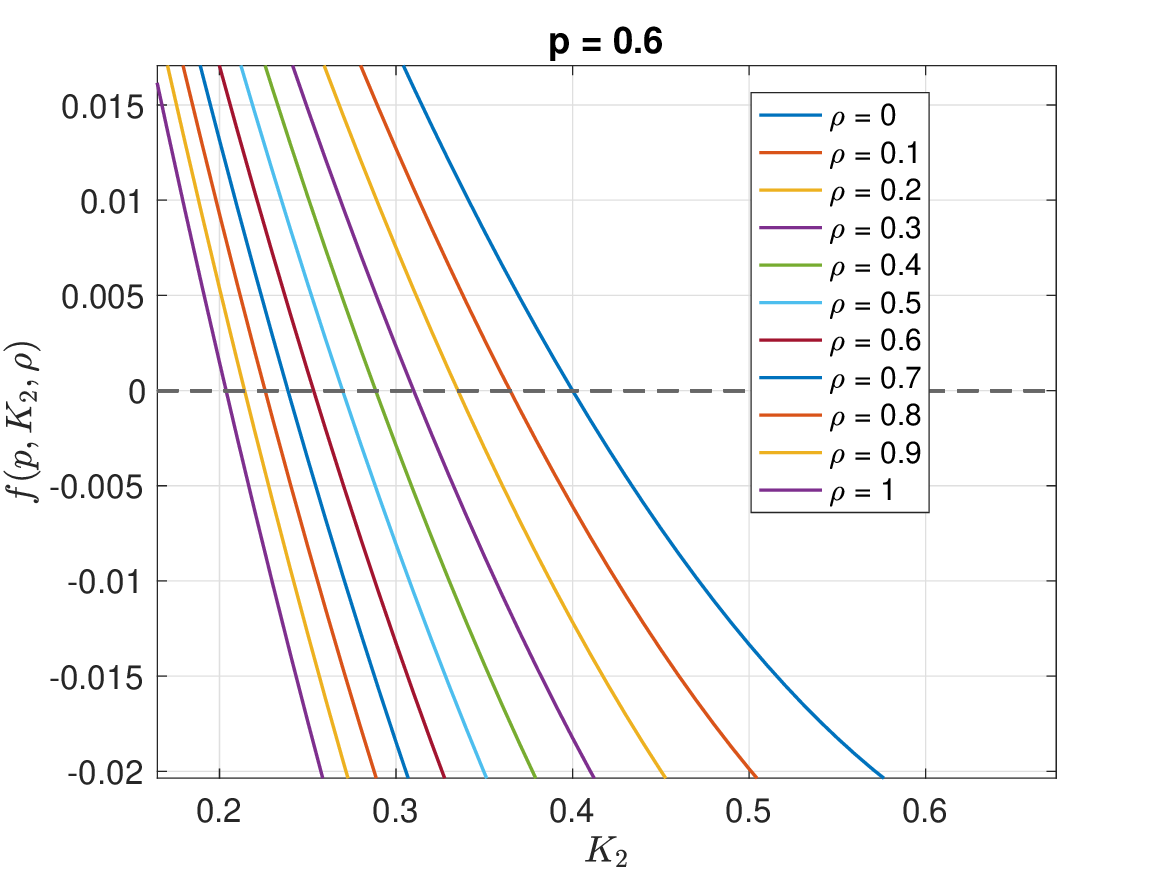}
    \caption{Zoom In: Optimal Feedback Gain $K_2^*$ for $\rho \in \{0, 0.1, 0.2, \dots, 1\}$ When $p=0.6$.}
    \label{fig: small p zoom in}
\end{figure}


\begin{figure}[h!]
    \centering
    \includegraphics[width=1\linewidth]{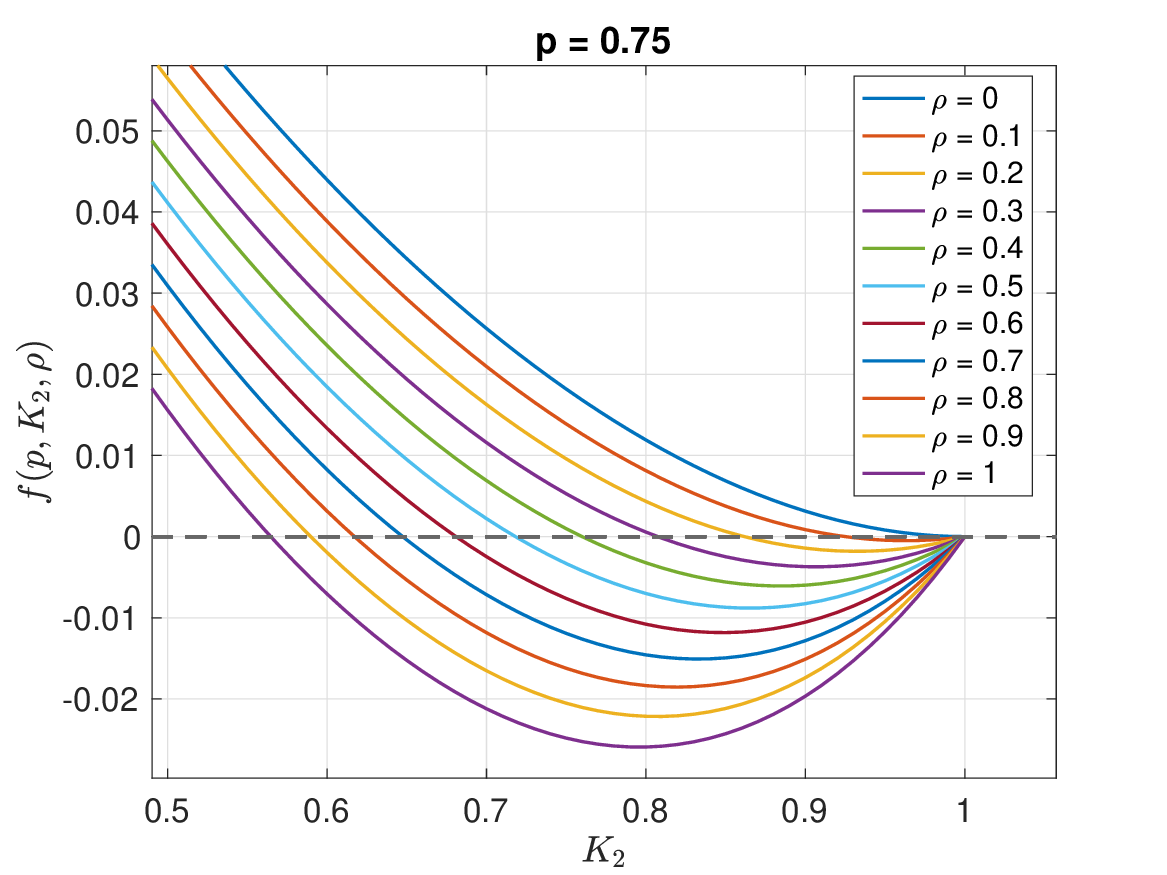}
    \caption{Zoom In: Optimal Feedback Gain $K_2^*$ for $\rho \in \{0, 0.1, 0.2, \dots, 1\}$ When $p = 0.75$.}
    \label{fig: large p zoom in}
\end{figure}

\begin{remark} \rm
    This example illustrates optimal betting and investment decisions and underlines its practical relevance. By illustrating how risk aversion $\rho$ influences the allocation of resources among different alternatives, we shed light on bridging theoretical models and real-world financial decision-making.  
\end{remark}

\subsection{Retail Inventory Management}  
   As a second example, we consider an inventory management problem faced by a retail operation, illustrating the application of a decision-making framework with an uncertain variable demand. 
    Specifically, the retail operation is confronted with the challenge of setting the inventory levels for two distinct categories, encapsulated by feedback gains $K_1$ and~$K_2:=1-K_1$ with~\(K_i \in [0, 1]\). These gains represent the proportion of the maximum possible inventory capacity \( I_{\max} \) units, for an upcoming sales period. These decisions are adjusted in anticipation of demand fluctuations, modeled as~$X_2(k) \sim {\rm Uniform}(-1, X_{\max})$ and $X_1(k) := 0$ with probability one. 
    
    The retailer's objective is to find the optimal feedback gain vector \(K\) that maximizes a risk-sensitive objective function~\eqref{problem: risk-sensitive problem}, integrating the expected logarithmic profits and the variance of these profits, with the risk aversion parameter~\( \rho \) modulating the influence of risk.

Similar to the prior example, we first verify the convexity of the log-variance. To this end, define an auxiliary random variable $\mathcal{X}_{n,i} := \mathcal{R}_{n,i} -1$ for $i \in \{1,2\}$ and $n\geq 1$. Note that
\begin{align*}
    v(K)
    &={\rm var}\left( \log \frac{V(n; K)}{V_0} \right) \\ 
     &= \mathbb{E}[(\log (1+K_2 \mathcal{X}_{n,2})^2] 
         - \mathbb{E}[\log (1+K_2 \mathcal{X}_{n,2})]^2\\
    & = \int_{-1}^{(1 + X_{\max})^n-1} (\log (1+K_2 x))^2 f_{\mathcal{X}_{n,2}} (x)dx \\
    &\qquad - \left( \int_{-1}^{(1+X_{\max})^n-1}  \log (1+K_2 x)f_{\mathcal{X}_{n,2}} (x)dx \right)^2.
\end{align*}
where $f_{\mathcal{X}_{n,2}}$ is the probability density function of induced payoff~$\mathcal{X}_n$ obtained in Proposition~\ref{proposition: Erlang compound return} in Appendix.

To illustrate, we take $I_{\max} = 1000$ units, symmetric demand volatility $X_{\max} = 1 = -X_{\min}$, risk-aversion constant~$\rho = 1/2$, and~$n \geq 1 $,
Figure~\ref{fig: Convexity of the Log-Variance} shows the second-order derivative of the log-variance under various $n \in \{1, 5, 10\}$, which is readily seen that $v(K)$ is a convex function; hence, Problem~\eqref{problem: risk-sensitive problem} is a concave program and Lemma~\ref{Lemma: risk term} becomes necessary and sufficient conditions. Moreover, by Lemma~\ref{Lemma: risk term}, with $n=5$, we obtain $K_1^*=1$ and $K_2^*=0$.

\begin{figure}[h!]
    \centering
    \includegraphics[width=1\linewidth]{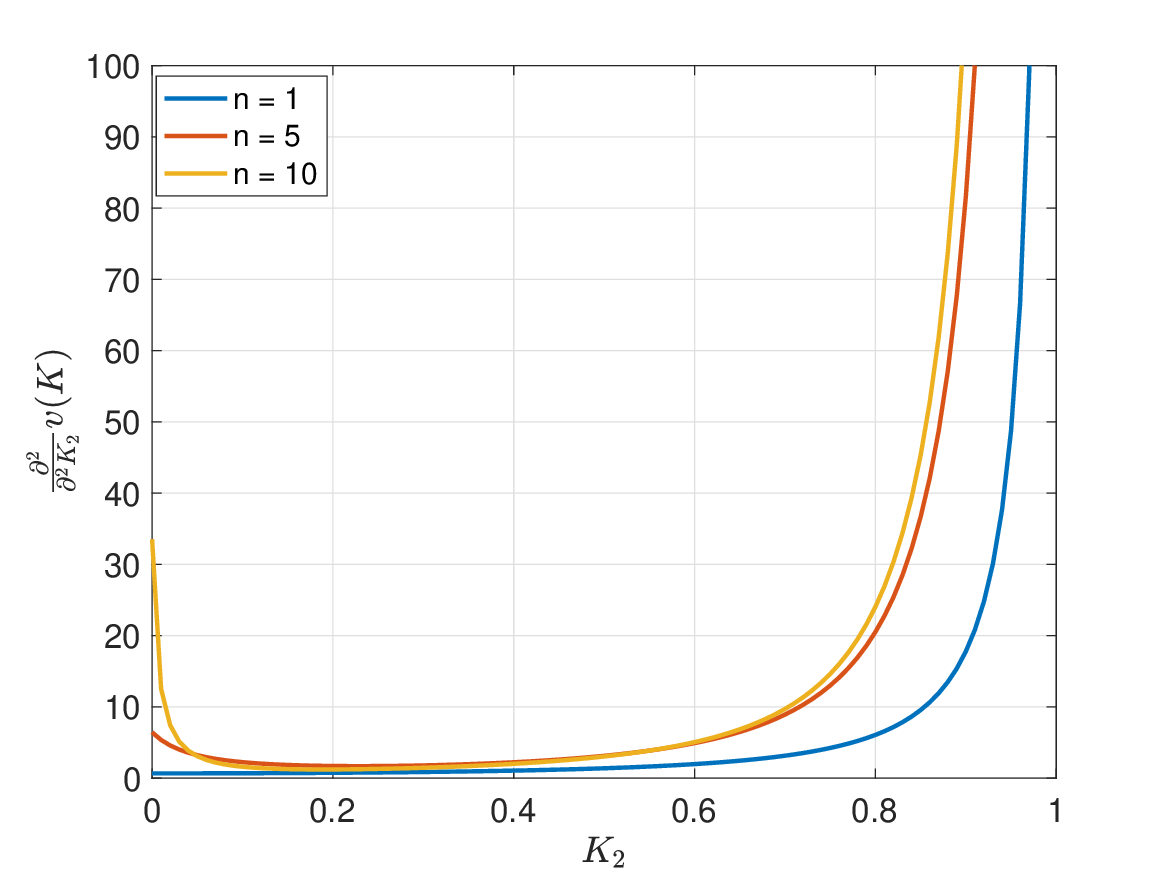}
    \caption{Convexity of the Log-Variance: $\frac{\partial^2 v(K)}{\partial K_2^2}$ Versus $K_2$.}
    \label{fig: Convexity of the Log-Variance}
\end{figure}

\begin{remark} \rm
    Leveraging a uniform distribution for initial demand fluctuations, which compound into an Erlang-based distribution, underscores the strategic advantage of predictive inventory management in uncertain environments. The optimization of \(K\) values based on this compounded demand model supports a {\em focused} policy, potentially prioritizing a single product category to manage risk and enhance returns effectively. 
\end{remark}

\section{Concluding Remarks} \label{section: Concluding Remarks}
    In this paper, we have introduced and examined a new risk-sensitive decision-making model under uncertainty, formulating the problem as a stochastic control problem and establishing necessary optimality conditions.
    Illustrative examples are given to demonstrate the idea.

    Our findings illuminate the intricate balance between maximizing expected log-growth and mitigating log-variance, providing a robust framework for decision-makers who face uncertain environments. This work has broad implications in financial sectors. Moreover, as seen in Section~\ref{section: Illustrative Examples}, the model's applicability extends beyond finance, offering valuable perspectives for risk management policies in sectors ranging from optimal betting to inventory management.

    Despite the contributions of this study, we recognize its limitations, including the i.i.d. assumptions on the stochastic alternatives. 
    For future work, one possibility is to explore more generalized models that relax these assumptions, for example, dive into the \textit{variance-regularized} optimization problem, see~\cite{duchi2019variance}, and seek possible convex surrogates of log-variance so that it is possible to enhance computational tractability and near-optimality.  
    Another interesting direction to pursue would be considering other practical constraints in the model such as risk limits~\cite{jorion2007value} or possible environmental and social constraints~\cite{eccles2014impact}.

 \appendix
\label{Induced Return:Uniform_Distribution}
This appendix studies the probability density function for induced payoffs
$\mathcal{X}_n := \prod_{k=0}^{n-1} (1+X(k)) -1 = \mathcal{R}_n - 1$
where $X(k)$ are i.i.d. uniform distributed random variables on~$[ -1, X_{\max}]$ with $X_{\max} > 0$.

\begin{proposition} \label{proposition: Erlang compound return}
    The probability density function for a random variable $\mathcal{X}_n := \prod_{k=0}^{n-1} (1+X(k)) -1  \in \mathbb{R}^1$ is as follows: For $ - 1 < z < \left( 1 + X_{\max } \right)^n - 1$, 
\[
f_{{\cal X}_n } \left( z \right) = \frac{1}{{{{(1 + {X_{\max }})}^n}\left( {n - 1} \right)!}}{\left( {\log \frac{{{{(1 + {X_{\max }})}^n}}}{{1 + z}}} \right)^{n - 1}};
\]	
Otherwise, $f_{\mathcal{X}_n}(z)=0.$
\end{proposition}

 The proof of the proposition above is established with the aid of the following lemma.

\begin{lemma} \label{lemma: auxiliary lemma}
    Let $X$ be a uniformly distributed random variable on $[-1, X_{\max}]$ with~$X_{\max} >0$. Define 
	$$
	Y := - \log \left( \frac{1 + X }{1+X_{\max}} \right).
	$$ 
 Then $Y$ is exponentially distributed with parameter $\lambda=1$; i.e., $Y \sim \exp(1)$. 
\end{lemma}

\begin{proof}
    We begin with computing the cumulative distribution function (CDF) for $Y$. Note that $F_Y(y) = 0$ for all $y < 0$; therefore, it suffices to consider the CDF for $y \ge 0 $; namely,
\begin{align*}
	F_Y(y)
	&= \mathbb{P} \left(  - \log \left( {\frac{{1 + X}}{1 + X_{\max } }} \right) \le y \right) 
	\\ 
	&= \mathbb{P}\left( {X \ge \left( {1 + {X_{\max }}} \right){e^{ - y}} - 1} \right)
	\\ 
	&= 1 - \mathbb{P}\left( {X \le \left( {1 + {X_{\max }}} \right){e^{ - y}} - 1} \right).
\end{align*}
Now note that for $y \ge 0$, 
$$ - 1 \le \left( {1 + {X_{\max }}} \right){e^{ - y}} - 1 \le {X_{\max }}$$ 
and since $X $ is uniform distributed on~$[-1, X_{\max}]$, we have
\begin{align*}
    F_Y(y)
    & = 1 - \frac{{\left( 1 + X_{\max} \right){e^{ - y}} - 1 + 1}}{X_{\max } + 1}\\
    &= 1 - e^{ - y},
\end{align*}
which is the CDF for an exponentially distributed with mean one; i.e., $Y \sim \exp(1).$
\end{proof}

 \begin{proof}[Proof of Proposition~\ref{proposition: Erlang compound return}] We begin with computing the cumulative distribution function~$F_{\mathcal{X}_n}(z)$ for $\mathcal{X}_n$. Note that~$F_{\mathcal{X}_n}(z)= 0$ for $z \le -1$; hence, it suffices to consider the CDF for~$z > -1$;~i.e.,
\begin{align*}
	F_{\mathcal{X}_n}(z)
	&= \mathbb{P}\left( {\prod\limits_{k = 0}^{n - 1} {(1 + X(} k)) - 1 \le z} \right)
	\\
	&= \mathbb{P}\left( {\sum\limits_{k = 0}^{n - 1} {\log (1 + X(k))}  \le \log \left( {1 + z} \right)} \right).
\end{align*}
Using the identity 
{\small 
\[
\sum_{k = 0}^{n - 1} {\log (1 + X(k))}  = \sum_{k = 0}^{n - 1} {\log \left( {\frac{{1 + X( k )}}{1 + X_{\max }}} \right)}  + n\log \left( {1 + {X_{\max }}} \right),
\]
}we have
{\small \begin{align*}
 &\mathbb{P} ({\cal X}_n \le z) \\
	&= \mathbb{P} \left( {\sum_{k = 0}^{n - 1} {\log \left( {\frac{{1 + X\left( k \right)}}{{1+X_{\max} }}} \right)}  + n\log \left( {1+X_{\max} } \right) \le \log \left( {1 + z} \right)} \right)
	\\
	&  = \mathbb{P} \left( {\sum_{k = 0}^{n - 1} Y (k) \ge n\log \left( 1 + X_{\max} \right) - \log \left( {1 + z} \right)} \right)
\end{align*}
}where
$ Y(k) := - \log \left( \frac{1 + X( k )}{1+X_{\max}} \right)
$
for $k=0,1, \dots ,n-1$.
Applying Lemma~\ref{lemma: auxiliary lemma}, it follows that $Y(k)\sim \exp(1)$. Furthermore, since the $Y(k)$ are i.i.d., using the fact that the sum of $n$ i.i.d. random variables with $\exp(1)$ is~$Erlang(n,1)$, see~\cite{gubner2006probability}, we have
$
E_n := \sum_{k=0}^{n-1} Y(k) \sim Erlang(n,1).
$
Now noting that $Erlang(n,1)$ has CDF $F_{E_n}(x) =0$ for $x \le 0$ and
$
F_{E_n}(x)= 1 - \sum_{k=0}^{n-1} \frac{x^k}{k!}e^{-x}
$  
for $x>0$. Hence, we have
\begin{align*}
	F_{\mathcal{X}_n}(z)
	& = \mathbb{P}\left( { \sum_{k=0}^{n-1} Y(k)   \ge n\log \left( {1 + {X_{\max }}} \right) - \log \left( {1 + z} \right)} \right)
	\\
	& = \mathbb{P}\left( {\sum_{k=0}^{n-1} Y(k)  \ge \log \left( {\frac{{{ {\left( 1 + X_{\max } \right) }^n}}}{{1 + z}}} \right)} \right)
	\\ 
	&= 1 - \mathbb{P}\left( {\sum_{k=0}^{n-1} Y(k)   \le \log  {\frac{{{{\left( 1+X_{\max} \right)}^n}}}{{1 + z}}} } \right).
\end{align*}
Since for $-1 < z < (1+X_{\max})^n -1 $, we have
\[\log \left( {\frac{{{{\left( {1 + {X_{\max }}} \right)}^n}}}{{1 + z}}} \right) > 0.
\]
Therefore, the corresponding CDF for $\mathcal{X}_n$ is given as follows: For $-1<z < {\left( {1 + X_{\max}} \right)^n} - 1$, we have
 $$
 F_{\mathcal{X}_n}(z ) =\frac{{1 + z}}{{{{\left( {1 + X_{\max}} \right)}^n}}}\sum_{k = 0}^{n - 1} {\frac{1}{{k!}}{{\left( {\log \left( {\frac{{{{\left( {1 + X_{\max}} \right)}^n}}}{{1 + z}}} \right)} \right)}^k}}
 $$
otherwise, $F_{\mathcal{X}_n}(z ) = 0.$
Thus, the associated probability density function is obtained by taking the derivative of the CDF above; i.e., For~${ - 1 < z < {{\left( {1 + {X_{\max }}} \right)}^n} - 1}$, we have
\[{f_{{\cal X}_n } }\left( z \right) = \frac{1}{{{{(1 + {X_{\max }})}^n}\left( {n - 1} \right)!}}{\left( {\log \frac{{{{(1 + {X_{\max }})}^n}}}{{1 + z}}} \right)^{n - 1}}\]	
otherwise, $f_{\mathcal{X}_n}(z)=0.$  
\end{proof}

\bibliographystyle{ieeetr}
\bibliography{refs}          

\end{document}